%% file: main.tex
\documentclass[a4paper,reqno,11pt]{amsart} 
\usepackage[T1]{fontenc}
\usepackage[utf8]{inputenc}
\newcommand{\q}[1]{``#1''}

\usepackage{amsmath, amssymb, graphicx, enumerate, enumitem, color}
\usepackage[usenames,dvipsnames,svgnames,table]{xcolor}
\usepackage[foot]{amsaddr}
\usepackage{todonotes}

\usepackage{multirow}
\usepackage{longtable}
\usepackage{booktabs}
\usepackage{comment}

\makeatletter
\def\thm@space@setup{
  \thm@preskip=4mm
  \thm@postskip=0mm
}
\makeatother

\usepackage{hyperref}
\usepackage{xcolor}
\hypersetup{
  pdfauthor={Jędrzej Hodor, William T. Trotter},
  pdftitle={Forcing the wheel},
    colorlinks,
    linkcolor={RoyalBlue},
    citecolor={RubineRed},
    urlcolor={blue!80!black}
}
\usepackage[capitalise]{cleveref}
\crefformat{figure}{Figure~#2#1#3}
\usepackage{mathtools}

\DeclarePairedDelimiter\cycl{\langle}{\rangle}

\theoremstyle{plain} 
\newtheorem{theorem}{Theorem}
\newtheorem{corollary}[theorem]{Corollary} 
\newtheorem{obs}[theorem]{Observation}

\newtheorem{conjecture}[theorem]{Conjecture}
\newtheorem{lemma}[theorem]{Lemma}
\newtheorem{remark}[theorem]{Remark}

\newtheorem{claim}[theorem]{Claim} 
\theoremstyle{remark} 
 
\bibstyle{plain}

\newcommand{\al}{\alpha}

\newcommand{\bbP}{\mathbb{P}}

\newcommand{\calP}{\mathcal{P}}
\newcommand{\calR}{\mathcal{R}}

\newcommand{\se}{\operatorname{se}}

\newcommand{\kelly}{\operatorname{kelly}} 
\newcommand{\wheel}{\operatorname{wheel}}

\graphicspath{{Figures/}}

\let\le\leqslant
\let\ge\geqslant
\let\leq\leqslant
\let\geq\geqslant

\let\subset\subseteq
\let\subsetneq\varsubsetneq

\let\epsilon\varepsilon

\renewenvironment{enumerate}{\begin{enumorig}[label=\textup{(\roman*)}, noitemsep, 
topsep=2pt plus 2pt, labelindent=.2em, leftmargin=*, widest=iii]}{\end{enumorig}}

\newenvironment{enumerateNumU}{\begin{enumorig}[label=\textup{(U\arabic*)}, 
noitemsep, topsep=2pt plus 2pt, labelindent=.2em, leftmargin=*, widest=iii]}{\end{enumorig}}

\newenvironment{enumerateNumV}{\begin{enumorig}[label=\textup{(V\arabic*)}, 
noitemsep, topsep=2pt plus 2pt, labelindent=.2em, leftmargin=*, widest=iii]}{\end{enumorig}}

\newenvironment{enumerateNumZ}{\begin{enumorig}[label=\textup{(Z\arabic*)}, 
noitemsep, topsep=2pt plus 2pt, labelindent=.2em, leftmargin=*, widest=iii]}{\end{enumorig}}

\makeatletter 
\let\old@setaddresses\@setaddresses 
\def\@setaddresses{\bigskip\bgroup\parindent 0pt\let\scshape\relax\old@setaddresses\egroup}
\makeatother

\begin{document} 
\title[FORCING THE WHEEL] 
{Forcing the wheel}

\author[J.~Hodor]{J\c{e}drzej Hodor}
\address[J.~Hodor]{Theoretical Computer Science Department\\ 
  Jagiellonian University\\ 
  Kraków, Poland}

\email{jedrzej.hodor@gmail.com}

\author[W.~T.~Trotter]{William T. Trotter}
\address[W.~T.~Trotter]{School of Mathematics\\ 
  Georgia Institute of Technology\\ 
  Atlanta, Georgia 30332}
\email{trotter@math.gatech.edu}

\thanks{J.\ Hodor is partially supported by a Polish National Science Center grant (BEETHOVEN; UMO-2018/31/G/ST1/03718).}
\thanks{W.\ T.\ Trotter is supported by a grant from the Simons Foundation.}

\date{\today}
                

\begin{abstract}
  Over the past 10 years, there has been considerable interest in exploring questions 
  connecting dimension for posets with graph theoretic properties of their cover graphs
  and order diagrams, especially with the concepts of planarity and treewidth.
  Joret and Micek conjectured that if
  $P$ is a poset with a planar cover graph, then the dimension of $P$ is bounded 
  in terms of the number of minimal elements of $P$ and the treewidth of the 
  cover graph of~$P$. 
  We settle this conjecture in the affirmative by strengthening a recent 
  breakthrough result~\cite{MSTB} by Blake, Micek, and Trotter, who proved that for 
  each poset $P$ admitting a planar cover graph and a unique minimal 
  element we have \mbox{$\dim(P) \leq 2 \se(P) + 2$}, namely, we prove that \mbox{$\dim(P) \leq 2 \wheel(P) + 2$}.
\end{abstract}

\maketitle

\section{Introduction}\label{sec:introduction}
\input{./s.introduction.tex}

\section{Background material}\label{sec:background-material}
\input{./s.background.tex}

\section{The Proof of Theorems \ref{thm:wheel} and \ref{thm:nxn-grid}}\label{sec:proofs}
\input{./s.proof.tex}

\section{Open problems}
\input{./s.open-problems.tex}\label{sec:open}

\bibliographystyle{abbrv}
\bibliography{main}
\end{document}

%% file: s.introduction.tex
We will assume that readers are familiar with basic concepts for posets 
including: subposets; chains and antichains;
height and width; maximal and minimal elements, comparable and incomparable
pairs of points; cover graphs and order diagrams; linear extensions;
realizers; and dimension as defined by Dushnik and Miller~\cite{DM41}.  
We will also assume that readers are familiar with basic 
concepts for graphs including: planarity; chromatic number; 
path-width; and treewidth. For readers who seek additional background on posets
and dimension, we recommend the recent papers~\cite{MSTB}, \cite{KMT21+} 
and the survey chapter~\cite{Tro19}.  All the graph theoretic material we use 
here can be found in any advanced undergraduate text on this subject.

Our primary goal in this paper is to resolve in the affirmative the following 
conjecture of Joret and Micek.

\begin{conjecture}\label{con:Joret-Micek}
  Among the class of posets with planar cover graphs, dimension is bounded
  in terms of the number of minimal elements and the treewidth of the cover
  graph.
\end{conjecture}

We will actually prove stronger results, and then
obtain the proof of Conjecture~\ref{con:Joret-Micek} as an immediate consequence.
The remainder of this introductory section contains material necessary to
motivate and state our three main theorems. In Section~\ref{sec:background-material},
we develop some more detailed background. The proofs of our three main
theorems are given in Section~\ref{sec:proofs}.  We close in 
Section~\ref{sec:open} with some remarks on open problems that remain.

When $P$ is a poset, the dimension of $P$ is denoted $\dim(P)$.  We will sometimes
abbreviate the statement $a\le b$ in $P$ as $a\le_P b$. Also, when $a$ and $b$
are incomparable elements of $P$, we write $a\parallel b$ in $P$, abbreviated
as $a\parallel_P b$.  When $P$ and $Q$ are posets, we will say that
$P$ \textit{contains} $Q$ if there is a subposet of $P$ that is isomorphic to
$Q$.  Dimension is monotonic, so $\dim(P)\ge\dim(Q)$ whenever $P$ contains $Q$.  We
will use the abbreviation $[n]$ for the set $\{1,\dots,n\}$ of the least $n$ 
positive integers.

When $P$ is a poset, the cover graph $G$ of $P$ is an undirected graph,
and there can be many different posets having $G$ as their 
cover graph.  Such posets can have different height, width, and dimension. Constructing examples illustrating the first two is straightforward, for dimension see e.g.\ \cite[Fig.\ 5]{Trotter_2015}. However, in this paper, we always start with a poset $P$,
and $P$ now imposes an orientation on $G$: An edge $e=xy$ in $G$ is oriented
from $x$ to $y$ when $y$ covers $x$ in $P$.  Also, we will say that $e$ 
\textit{leaves} $x$ and \textit{enters}~$y$.  With this convention, we
can specify a poset with a drawing of its oriented cover graph, and this
method will be used extensively in the balance of the paper.

A poset is \textit{planar} if its order diagram can be drawn without edge crossings
in the plane.  We say that a poset is \textit{cover-planar} if its cover graph is
planar.  A planar poset is also cover-planar, but there are cover-planar
posets that are not planar.  An infinite family of such posets will
be discussed later in this section. 

Generally, in the past 15 years alone, more than 25 published research papers have 
explored connections between dimension and graph theoretic
properties of cover graphs and order diagrams (we list some of them: outerplanarity \cite{Felsner_2014,GS21+}, planarity \cite{Streib_2014,KMT21+,BBSTWY21,JMW17}, cut-vertices structure \cite{trotter2018dimension}, treewidth \cite{SEWERYN2020111605,JMMTWW16}, exclusion of some structures \cite{Wal17,MW17,HSTWW19,Huynh_2021}, or even more complex structural parameters \cite{JMW18,JMOdMW19}). For our considerations, it is worth mentioning a result in \cite{JMMTWW16}. Joret et.\ al.\ proved that among cover-planar posets the dimension is bounded by a function of their height and treewidth. Note a similarity to our result -- we replace the height with the number of minimal elements. 

However, to state our results, first, we have to explore the following well-known statement: The dimension of (cover-)planar posets is not bounded. In order to discuss this in detail, we introduce the family of standard examples. That is, for each $d \ge 2$ the \emph{standard~example~of order~$d$}, sometimes denoted by $S_d$, is a height~$2$ poset with minimal
elements $\{a_1,\dots,a_d\}$ and maximal elements $\{b_1,\dots,b_d\}$.  The order
relation is $a_i<b_j$ in $S_d$ when $i,j\in[d]$ and $i\neq j$. As noted
in~\cite{DM41}, $\dim(S_d)=d$ for all $d\ge2$.  We show an order diagram 
of the standard example $S_6$ in part~\textbf{A} of \cref{fig:examples}.
Note that $S_d$ is planar when $d\le4$.  However, when $d\ge5$, even the 
cover graph of $S_d$ is not planar.  

In~\cite{Tro78}, a family of cover-planar posets with unbounded dimension is constructed, namely, for each $d \geq 3$ there is a cover-planar poset $H_d$ such that $\dim(H_d)=d$. In several recent papers, $H_d$ is called the \emph{wheel of order $d$}. We define the wheels\footnote{We 
elect to include the unique minimal element but not the unique maximal element 
in our definition of a wheel.} formally in the next section, see a plane drawing of an oriented cover graph for the wheel $H_6$ in part \textbf{B} of \cref{fig:examples}.
The poset $H_d$ contains the standard example $S_d$, so $\dim(H_d)\ge d$. It is an easy exercise to show that $\dim(H_d)=d$. 
Moreover, the cover graph $H_{2n+1}$ contains an $n \times n$ grid as a subgraph, and so it has treewidth at least $n$.
Note that each wheel has a unique minimal element.
Finally, for each $d \geq 4$, the wheel $H_d$ is not a planar poset. This follows from the fact proved in \cite{TM77} stating that a planar poset with a unique minimal element has dimension bounded by $3$.

In~\cite{Kel81}, a family of planar posets with unbounded dimension is constructed, namely, for each $d \geq 3$ there is a planar poset $K_d$ such that $\dim(K_d)=d$. The poset $K_d$ is usually called the \emph{Kelly poset of order $d$}.
We show a plane drawing of an order diagram for the Kelly poset $K_6$ in part \textbf{C} of \cref{fig:examples}.
The poset $K_d$ contains the standard example $S_d$, so $\dim(K_d)\ge d$. Again, it is an easy exercise to show that $\dim(K_d)=d$.
For each $d \geq 3$ the wheel of order $d$ contains the Kelly poset of order $d$, however, the two construction differ a lot in terms of properties.
The cover graph of any Kelly poset has pathwidth (and so treewidth) bounded by~$3$ -- see e.g.~\cite{JMMTWW16}.
Moreover, Kelly posets of large order have many minimal elements.

\begin{figure} 
  \begin{center}
    \includegraphics{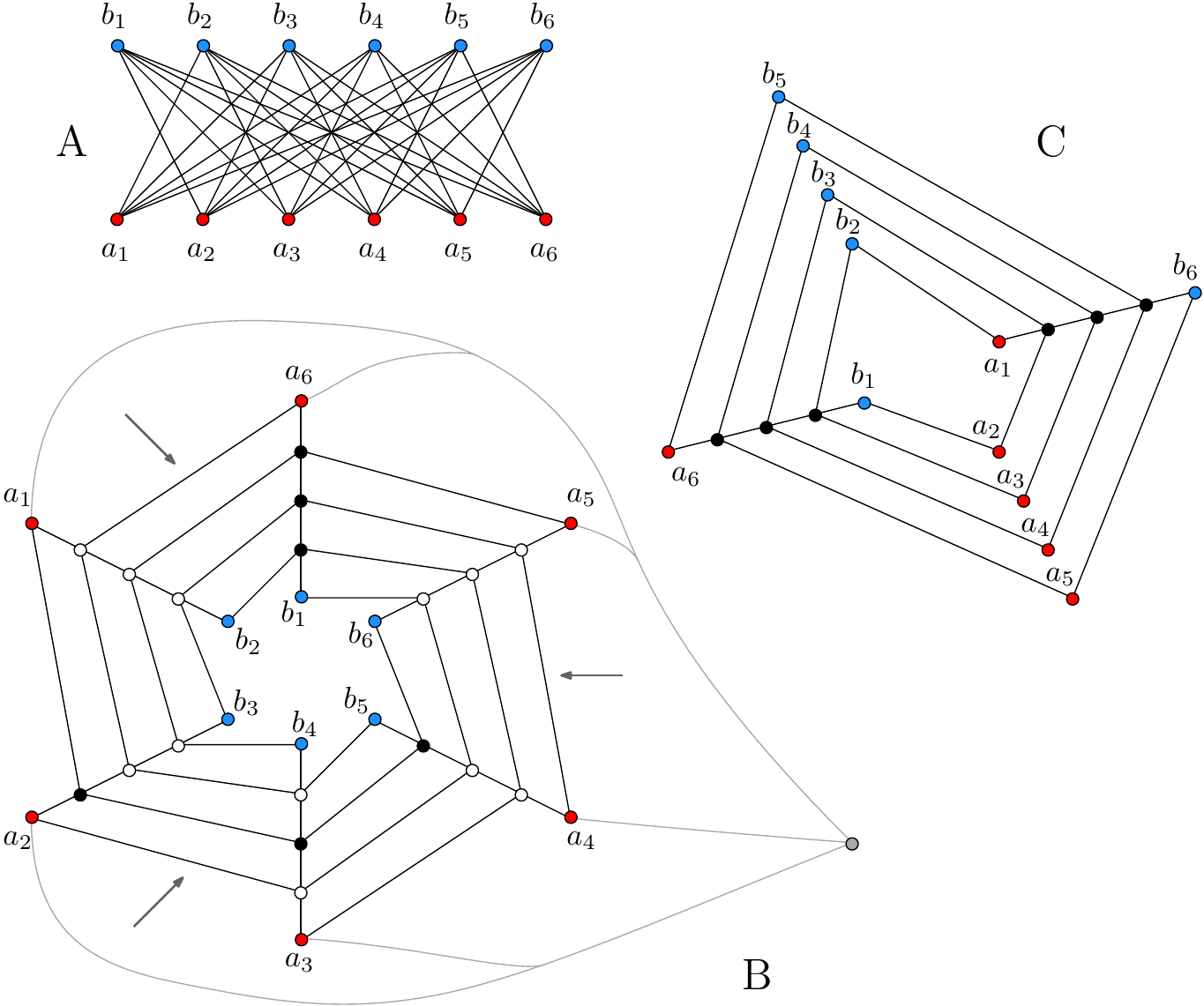} 
  \end{center}
  \caption{
    \newline\textbf{A.}
    The order diagram of the standard example $S_6$.
    \newline\textbf{ B.}
    The oriented cover graph of the wheel $H_6$, a cover-planar 
    poset with a unique minimal element. The gray element is a unique minimal element, all edges are directed \q{inside} the wheel as the gray arrows indicate. Red and blue elements form the standard example $S_6$. Red, blue, and black elements form the Kelly poset $K_6$.
    \newline\textbf{C.}
    The order diagram of the Kelly poset $K_6$. One may find the Kelly poset in literature in forms that vary a little from the one that we present. Usually, this is only due to aesthetic reasons.} 
  \label{fig:examples}
\end{figure}

For convenience, we introduce the following parameters measuring complexity of posets. For a poset $P$, the \emph{standard example number}
of $P$, denoted $\se(P)$, is set to be $1$ if $P$ does not contain the standard
example $S_2$; otherwise, $\se(P)$ is the maximum order of a standard example contained in $P$. The \emph{wheel number} (or \emph{Kelly number} respectively) denoted by $\wheel(P)$ (or $\kelly(P)$) of a poset $P$ is equal to $\se(P)$ if there is no wheel (or no Kelly poset) contained in $P$; otherwise, it is equal to the maximum order of a wheel (or a Kelly poset) contained in $P$. Note that the standard example number is a classical notion, whereas the two latter parameters are much less studied. Combining some of the observations stated up to this point, for each poset $P$ we have $\wheel(P) \leq \kelly(P) \leq \se(P) \leq \dim(P)$.

The last inequality can be far from tight. The family of posets that do not contain the standard example $S_2$ as a subposet is
a well-studied class of posets called \textit{interval orders}.  Posets
in this class can have arbitrarily large dimension~\cite{FHRT91}. Clearly, it is an interesting property of a family of posets when the dimension and the standard example number are functionally bonded. 
We say that a class $\bbP$ of posets is \textit{$\dim$-bounded}
if there is a function $f$ such that $\dim(P)\le f(\se(P))$ for every $P\in\bbP$. This reflects the notion of $\chi$-boundedness in graph theory, we refer readers to the recent survey by Scott and Seymour~\cite{SS2020} on
the extensive body of research done on this topic. 

We discussed the fact that the dimension of (cover-)planar posets is not bounded, thus it is natural to ask for $\dim$-boundedness.
The roots of the following conjecture can be traced back more than~$40$ years,
although the first reference in print is in the form of an
informal comment on page~119 in~\cite{Tro-book} published in 1992.

\begin{conjecture}\label{con:pcg-db}
  The class of cover-planar posets is $\dim$-bounded.
\end{conjecture}

Recently, Blake, Micek, and Trotter resolved this conjecture
for cover-planar posets with a unique minimal element by
proving the following theorem.  

\begin{theorem}\cite{MSTB}\label{thm:pcg0-db}
  If $P$ is a cover planar poset, and $P$ has a unique minimal element, then
  \[
    \dim(P)\le 2\se(P)+2.
  \]
\end{theorem}

The following more general result is an immediate corollary.

\begin{corollary}\label{cor:pcgt-db}
  If $P$ is a cover planar poset, and $P$ has $m$ minimal elements, then
  \[
    \dim(P)\le m(2\se(P)+2).
  \]
\end{corollary}

Conjecture~\ref{con:Joret-Micek} was made soon after the results 
of~\cite{MSTB} were announced.  We can now state our three main theorems.

\begin{theorem} \label{thm:wheel}
  Let $P$ be a cover-planar poset with a unique minimal element.  Then
  \[
    \dim(P)\le 2\wheel(P)+2.
  \]
\end{theorem}

\begin{theorem} \label{thm:nxn-grid}
  Let $P$ be a cover-planar poset with a unique minimal element.  If $n\ge2$,
  and $\dim(P)$ is at least $4n+3$, then the cover graph $G$ of $P$ contains an 
  $n\times n$ grid as a minor. In particular, $\dim(P) \leq 4\mathrm{tw}(G)+6$
\end{theorem}

\begin{theorem} \label{thm:minimal-tw}
  Let $P$ be a cover-planar poset.  If $P$ has $m$ minimal elements and $G$ is the cover graph of $P$, then
  \[
    \dim(P)\le m(4\mathrm{tw}(G)+6).
  \]
\end{theorem}

\cref{thm:minimal-tw} is almost a straightforward corollary of \cref{thm:nxn-grid}. This is due to the fact that the treewidth of the $n \times n$ grid is $n$ and the following easy observation.
\begin{obs}\cite[Proposition 3.2]{Trotter_2015} \label{obs:extensions}
    Let $P$ be a poset and $x \in P$. Let $U$ be the upset of $x$ in~$P$, that is, $U = \{y \in P : x \leq_P y\}$. If $L$ is a linear extension of the poset induced on~$U$ in~$P$, then there exists a linear extension $L'$ of $P$ such that
        \begin{itemize}
            \item for all $y_1,y_2 \in U$ if $y_1 \leq y_2$ in $L$, then $y_1 \leq y_2$ in $L'$, and 
            \item for all $z \in P - U$ and $y \in U$ we have $z < y$ in $L'$.
        \end{itemize}
\end{obs}

Let us make a brief note on the relation between the dimension and the number of minimal elements in the case, where a poset has a planar diagram. As already mentioned, in \cite{TM77} it is proved that for a planar poset $P$ with a unique minimal element, we have $\dim(P) \leq 3$. This result combined with \cref{obs:extensions} gives that for every planar poset $P$ with $m$ minimal elements we have $\dim(P) \leq 3m$. In \cite{Trotter_2015} the inequality is strengthened to $\dim(P) \leq 2m+1$. It is not known, whether this is tight. The best known lower bound appeared in the same paper. For each $m \geq 3$, the authors constructed a planar poset $P_m$ with $m$ minimal elements such that $\dim(P) \geq m + 3$.

Let $h(P)$ denote the height of a poset $P$. Observe that for every poset $P$ we have $\wheel(P) \leq h(P)$. Therefore, \cref{thm:wheel} implies the following.
\begin{corollary}\label{cor:height}
    Let $P$ be a cover-planar poset with a unique minimal element. Then
        \[\dim(P) \leq 2h(P)+2.\]
\end{corollary}
The assumption on a unique minimal element seems to simplify a lot, as in general, it is only known that $\dim(P) \leq c h^6(P)$ for every cover-planar poset $P$ and some absolute constant $c$ (the first bound was given in \cite{Streib_2014}, later improved in~\cite{KMT21+}). In the case of posets with a planar diagram, the best known bound is linear, however, with a much worse multiplicative constant, namely, we have $\dim(P) \leq 192 h(P) + 96$ for every poset $P$ with a planar diagram~\cite{JMW17}.

%% file: s.background.tex
We use the symbol $:=$ to underline that some object is defined. For natural numbers $a,b$ we write $[a,b]$ for the set $\{a,a+1,\dots,b\}$. To simplify, we abbreviate $[a] := [1,a]$. 
For a given positive number $N$ and $i,j \in [N]$ let $\cycl{i,j} := [i,j]$ if $i \leq j$ and $\cycl{i,j} := [j,N] \cup [1,i]$ otherwise. We call $\cycl{i,j}$ a \emph{cyclic interval}. The number $N$ will always be clear from the context. 

We have already shown a drawing of the cover graph of the wheel $H_6$ in part \textbf{B} of \cref{fig:examples}. Now, we will define the family of the wheels in a formal manner.
For a natural number $N \geq 3$, the wheel of order $N$ is a poset with the ground set $\{r_{i,j} : i,j \in [N], \ j+1 \not\equiv i \mod N\} \cup \{\textrm{min}\}$. We have $r_{i,j} \leq r_{i',j'}$ in the poset if and only if $\cycl{i',j'} \subset \cycl{i,j}$, and moreover, min is less than every other element.
Note that the poset induced by the set $\{r_{i+1,i-1} : i \in [N]\} \cup \{r_{i,i} : i \in [N]\}$ is the standard example of order $N$. See \cref{fig:wheel} for an illustration.

\begin{figure} 
  \centering 
  \includegraphics{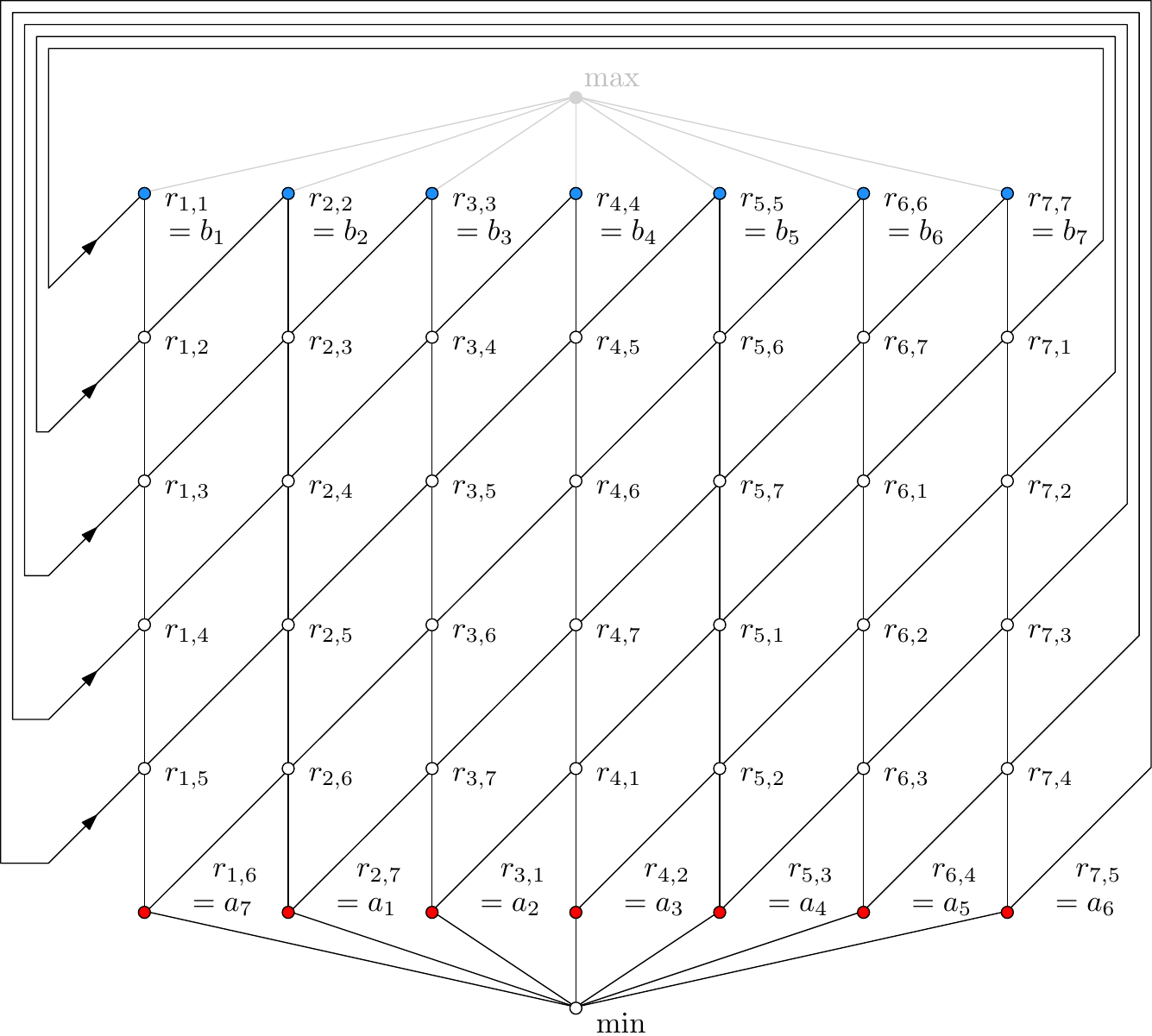} 
  \caption{A planar drawing of the cover graph of the wheel $H_7$ along with the natural orientation. All comparabilities go upward unless specified otherwise by an arrow. The red and blue elements form a standard example. An element $r_{i,j}$ is below the elements $b$ exactly with indices in the interval $\cycl{i,j}$. Note that one can also attach a unique maximal element.} 
  \label{fig:wheel} 
\end{figure} 

In \cite{MSTB} the authors showed that if a cover-planar poset with a unique minimal element has large dimension, then it has a large standard example as a subposet. 
The proof relies on a deep understanding of the structure of such posets. A few reductions allow us to \q{find} a standard example in a much simpler setting than the general one. 
We are going to show that this standard example is embedded in the wheel.
To state the mentioned reductions, we need a few definitions from \cite{MSTB} and \cite{KMT21+}.

Let us fix a cover-planar poset $P=(X,\leq_P)$ with a unique minimal element $x_0$, and let us fix a plane drawing of its cover graph $G$ with $x_0$ on the exterior face. 
We attach a one-end edge $e_{\infty}$ to $x_0$ in such a way that the edge is contained in the exterior face. The special edge is directed towards $x_0$, that is, $e_\infty$ enters $x_0$.

A \emph{path} in the graph $G$ is a tuple $W = (u_0,u_1,\dots, u_s)$ of pairwise distinct vertices such that for each $i \in [s]$ either $u_{i-1}$ covers $u_i$ in $P$ or $u_{i}$ covers $u_{i-1}$ in $P$. For every $0 \leq i \leq j \leq s$, we denote by $u_i[W]u_j$ the path $(u_i,\dots,u_j)$.
Let $x,y \in X$ be such that $x \leq_P y$. A path $W = (u_0,u_1,\dots,u_{s})$ in $G$ is called a \emph{witnessing path from $x$ to $y$} if $u_0 = x, u_{s} = y$, and for each $i \in [s]$ we have $u_{i-1} <_P u_i$.

Let $u \in X$ and $e$ be an edge entering $u$ in $G$. There is a natural left to right (clockwise) order of the edges leaving $u$: start with $e$ and go clockwise around $u$ following the plane drawing of $G$. We will refer to this ordering as \emph{$(u,e)$-ordering}. See part \textbf{A} of \cref{fig:natural-order}.

\begin{figure} 
  \centering 
  \includegraphics[scale=1]{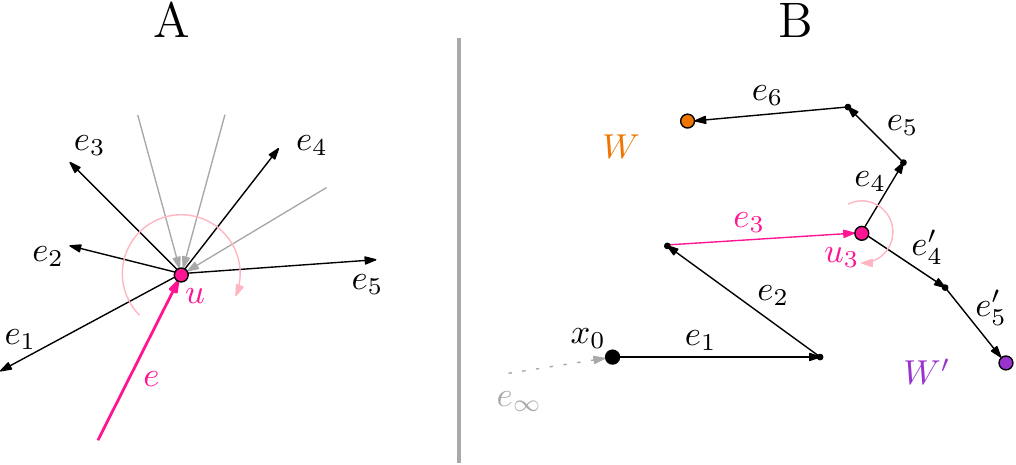} 
  \caption{\newline\textbf{A.} The $(u,e)$-ordering of the edges leaving $u$ is $e_1<e_2<e_3<e_4<e_5$.
\newline
  \textbf{B.}
  The witnessing paths $W$ and $W'$ start together and split in the vertex $u_3$.
  The edge $e_4$ is left of the edge $e_4'$ in the $(u_3,e_3)$-ordering, thus the path $W$ is left of the path $W'$.
  } 
  \label{fig:natural-order} 
\end{figure} 

Let $u \in X$. 
The \emph{leftmost witnessing path from $x_0$ to $u$}, denoted by $W_L(u)$, is constructed using the following inductive procedure. 
Start with $u_0 := x_0$ and $e_0 := e_{\infty}$. We have $u_0 \leq_P u$ and $e_0$ enters $u_0$. 
Suppose that for some $i \geq 0$ we have defined a witnessing path $W_i = (u_0,\dots, u_i)$ where $u_i \leq_P u$ and an edge $e_i$ entering $u_i$.
If $u_i = u$, then $W_L(u) := W_i$.
Otherwise, let $E_i$ be the set of all edges $e = u_iv$ leaving $u_i$ such that there exists a witnessing path from $v$ to $u$.
It follows that $E_i$ is nonempty.
We define $e_{i+1}$ as the leftmost edge in the set $E_i$ with respect to the $(u_i,e_i)$-ordering. Next, we define $u_{i+1}$ to be the other endpoint of $e_{i+1}$.
In an analogous way, we define the \emph{rightmost witnessing path from $x_0$ to $u$} denoted by $W_R(u)$. See part \textbf{A} of \cref{fig:witnessing-paths} for an example.
Such witnessing paths satisfy many natural properties. 
For instance, if $u,u' \in X$, then the paths $W_L(u)$ and $W_L(u')$ start together, and once they split they never meet again (analogous claim holds for the paths $W_R(u)$ and $W_R(u')$). 
For a detailed and technical analysis see \cite[Section~3]{MSTB}. 

Let $W = (u_0,\dots,u_s)$ and $W'=(u_0',\dots,u_{s'}')$ be two witnessing paths from $x_0$ to some elements in $P$ such that none is a prefix of the other. Let $e_0 = e_0' = e_\infty$, for each $i \in [s]$, let $e_i := u_{i-1}u_i$, and for each $i \in [s']$, let $e_i' := u_{i-1}'u_{i}'$. Let $i$ be the minimal index such that $e_i \neq e_i'$, clearly $i>0$. We say that $W$ is \emph{left of} (\emph{right of}) $W'$ if $e_{i}$ is left of (right of) $e_i'$ in the $(u_{i-1},e_{i-1})$-ordering, see part \textbf{B} of \cref{fig:natural-order}.

We say that an incomparable pair $(a,b)$ is \emph{a left pair} if $W_L(a)$ if left of $W_L(b)$ and $W_R(a)$ is left of $W_R(b)$. 
An incomparable pair $(a,b)$ is \emph{a right pair} if $W_L(a)$ if right of $W_L(b)$ and $W_R(a)$ is right of $W_R(b)$. See part \textbf{B} 
 of \cref{fig:witnessing-paths}.
\begin{figure} 
  \centering 
  \includegraphics{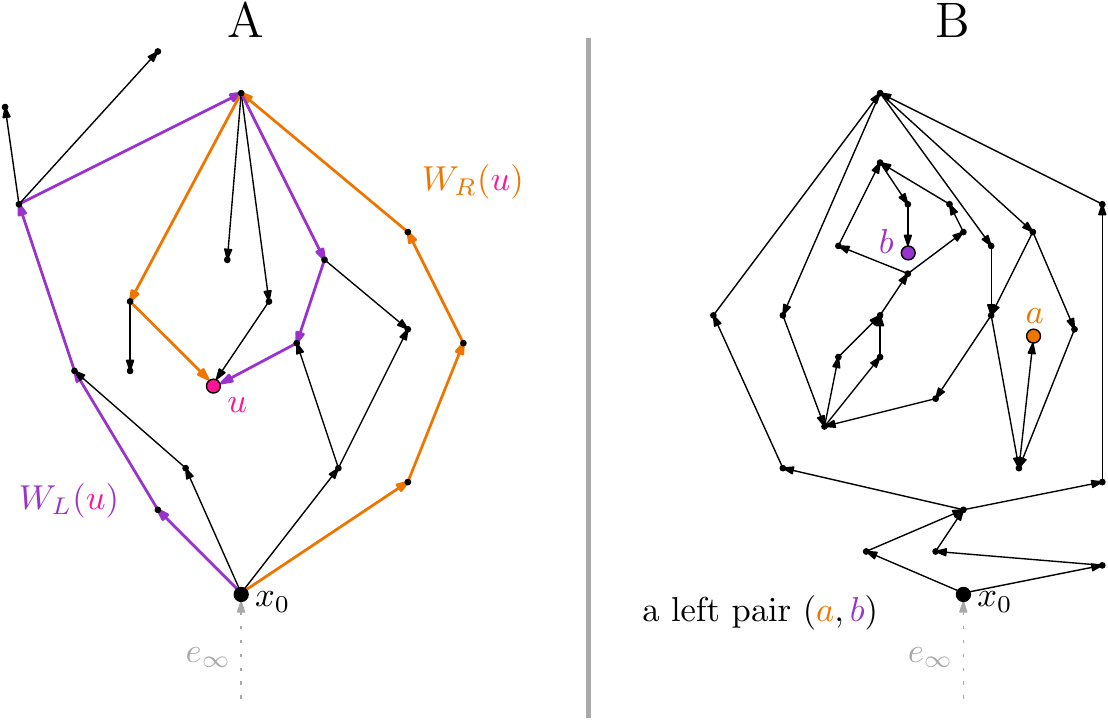} 
  \caption{\newline\textbf{A.} An example of $W_L(u)$ and $W_R(u)$ for some element $u \in X$. Note that witnessing paths can \q{turn back}. In this figure only once, however, in general arbitrarily many times. 
  \newline\textbf{B.} An example of a left pair $(a,b)$.} 
  \label{fig:witnessing-paths} 
\end{figure} 

For two given elements $x,y \in X$ with $x \leq_P y$ and two witnessing paths $W$ and $W'$ from $x$ to $y$ such that the only common elements of $W$ and $W'$ are $x$ and $y$, the \emph{$(x,y,W,W')$-interval} is the region on the plane enclosed by $W$ and $W'$ that does not contain $e_{\infty}$. 
A subposet \emph{induced by an interval} is a subposet induced by all elements in the interval. 
An interval has a \emph{shadowing property} if for every element $z \in X$ lying in the interval and for each $D \in \{L,R\}$ the path $W_D(z)$ contains $x$ and the path $x [W_D(z)] z$ is contained in the interval. Let $Q$ be induced by the interval. It follows that $x$ is a unique minimal element of $Q$. Note that the planar drawing of $P$ induces a planar drawing of $Q$ with $x$ on the exterior face. The shadowing property of the interval gives that for every $x,y$ in the interval such that $(x,y)$ is a left (or right) pair in $P$ we have that $(x,y)$ is a left (or right) pair in $Q$. We will use this observation implicitly many times.

We are ready to state the result from \cite{MSTB}, which is a starting point for our proof. Recall that the main result of \cite{MSTB} is dim-boundedness of cover-planar posets with unique minimal elements. The next lemma is a stronger statement, namely, it describes \q{how} a poset contains a large standard example in the considered setting, see part \textbf{A} of \cref{fig:sep-paths}. Using this particular description, we are going to prove that this standard example forces the wheel as a subposet, see \cref{fig:final}.

\vbox{
    \begin{lemma} \cite[Proposition~12, Corollary~35]{MSTB} \label{lem:main}
    Let $k \geq 2$ and $P=(X,\leq_P)$ be a cover-planar poset with a unique minimal element $x_0$ and $\dim(P) \geq 2k+1$. Fix a plane drawing of the cover graph of $P$ with $x_0$ on the exterior face. There exist two elements $x,y \in X$ with $x \leq_P y$, two witnessing paths $W,W'$ from $x$ to $y$, and $a_1,\dots,a_k,b_1,\dots,b_k \in X$ such that:
    \begin{enumerate}
        \item the $(x,y,W,W')$-interval satisfies the shadowing property,\label{lem:main:simple:interval}
        \item $a_1,\dots,a_k,b_1,\dots,b_k$ are in the interval and induce the standard example of order $k$ in $P$,\label{lem:main:se}
        \item we have $a_\al \parallel_P y$ and $y <_P b_\alpha$ for each $\alpha \in [k]$, \label{lem:main:see:structure}
        \item $(a_\alpha,a_\beta)$ and $(b_\alpha,b_\beta)$ are left pairs for all $1 \leq \alpha < \beta \leq k$.\label{lem:main:see:structure:left:right}
    \end{enumerate}
\end{lemma}
}

\begin{figure} 
  \centering 
  \includegraphics[scale=1]{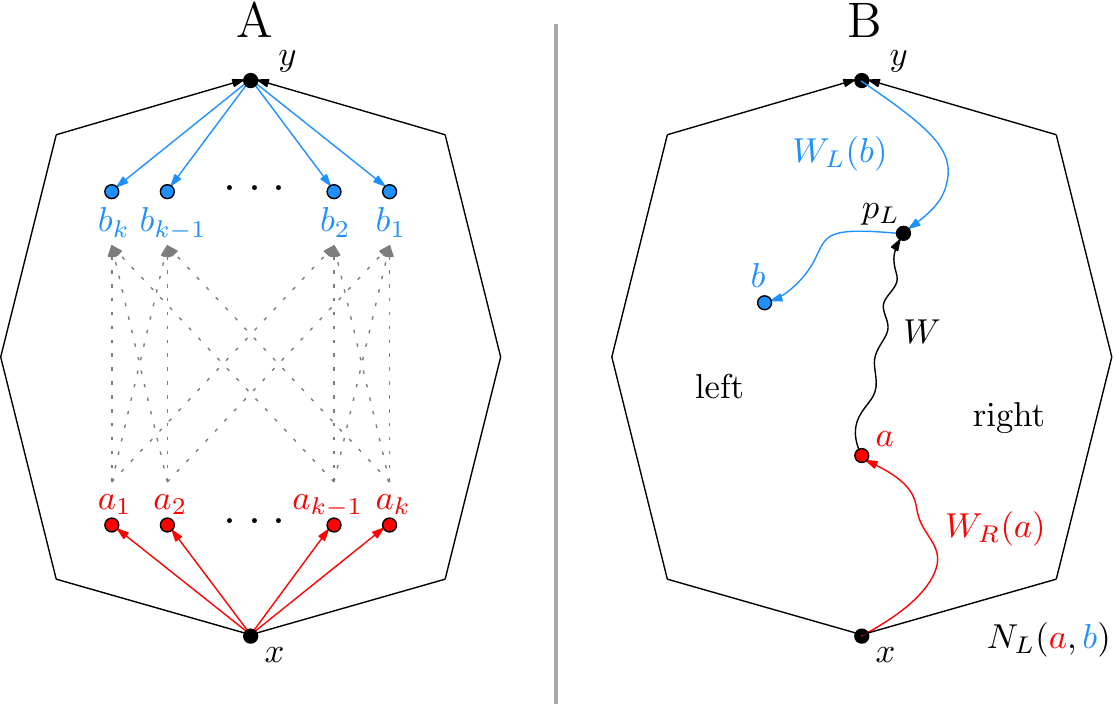} 
  \caption{\newline\textbf{A.} An illustration of the assertion of \cref{lem:main}. There is an interval along with a standard example within. The elements $a$ are in the left-to-right order as well as the elements $b$ (note that $(b_1,b_2)$ is a left pair!). The comparabilities in the standard example are drawn with dotted grey lines. Most of our proof focuses on introducing structure to these comparabilities. 
  \newline\textbf{B.} An example of a left separating path $N_L(a,b)$. It starts with a segment of the leftmost witnessing path to $a$, then it turns into some path witnessing $a <_Q p_L$, and finally, it goes backward with the leftmost witnessing path to $b$.}
  \label{fig:sep-paths} 
\end{figure} 

Suppose that a poset $P$ satisfies the assumptions of the above lemma. Let $x,y \in X$ and witnessing paths $W,W'$ be such that items~\ref{lem:main:simple:interval}--\ref{lem:main:see:structure:left:right} are satisfied. Let $Q=(Y,\leq_Q)$ be a poset induced by the $(x,y,W,W')$-interval. From now on, we are going to restrict our attention to the poset $Q$. We define 
    \[\widehat{A} := \{ z \in Y : y \parallel_Q z \}, \ \ \widehat{B} := \{z \in Y : y <_Q z\}, \ \ \textrm{and} \ \ \widehat{E} := \{ z \in Y :  x \leq_Q z \leq_Q y \}.\]
    Note that $Y = \widehat{A} \cup \widehat{B} \cup \widehat{E}$.
Fix $(a,b) \in \widehat{A} \times \widehat{B}$ such that $a <_Q b$. \mbox{Let $p_L$ (resp.\ $p_R$)} be the least element on $W_L(b)$ (resp.\ $W_R(b)$) such that $a <_Q p_L$ (resp.\ $a <_Q p_R$).
Let $W$ be some witnessing path from $a$ to $p_L$ (resp.\ $p_R$). We define a \emph{left separating path $N_L(a,b)$ associated with $a$ and $b$} as the path (not necessarily a witnessing path!) that consists of the segments: (see part \textbf{B} of \cref{fig:sep-paths})
    \[ x[W_R(a)]a, \ \ a[W]p_L, \ \ y[W_L(b)]p_L. \]
We define a \emph{right separating path $N_R(a,b)$ associated with $a$ and $b$} as the path that consists of the segments
    \[ x[W_L(a)]a, \ \ a[W]p_R, \ \ y[W_R(b)]p_R. \]
The element $p_L$ (resp.\ $p_R$) is called the \emph{peak} of a path $N_L(a,b)$ (resp.\ $N_R(a,b)$). Given a separating path in an interval, we can naturally say that some objects are on the \emph{left} or on the \emph{right} of the path. Let us finish this section with one simple observation on separating paths, which somehow justifies the name.
\begin{obs}\label{obs:separating-paths}
    Let $P,x,y,W,W',Q,\widehat{A},\widehat{B}$ be as above. Let $a,a' \in \widehat{A},b\in \widehat{B}$ be such that $a<_Q b$ and $(a,a')$ is an incomparable left pair. For every left or right separating path $N$ associated with $a$ and $b$ ($N_L(a,b)$ or $N_R(a,b)$) the element $a'$ is on the right of the path $N$. Symmetrically, if $(a,a')$ is an incomparable right pair, then the element $a'$ is on the left of the path $N$.
\end{obs}
\begin{proof}

\begin{figure} 
  \centering 
  \includegraphics[scale=1]{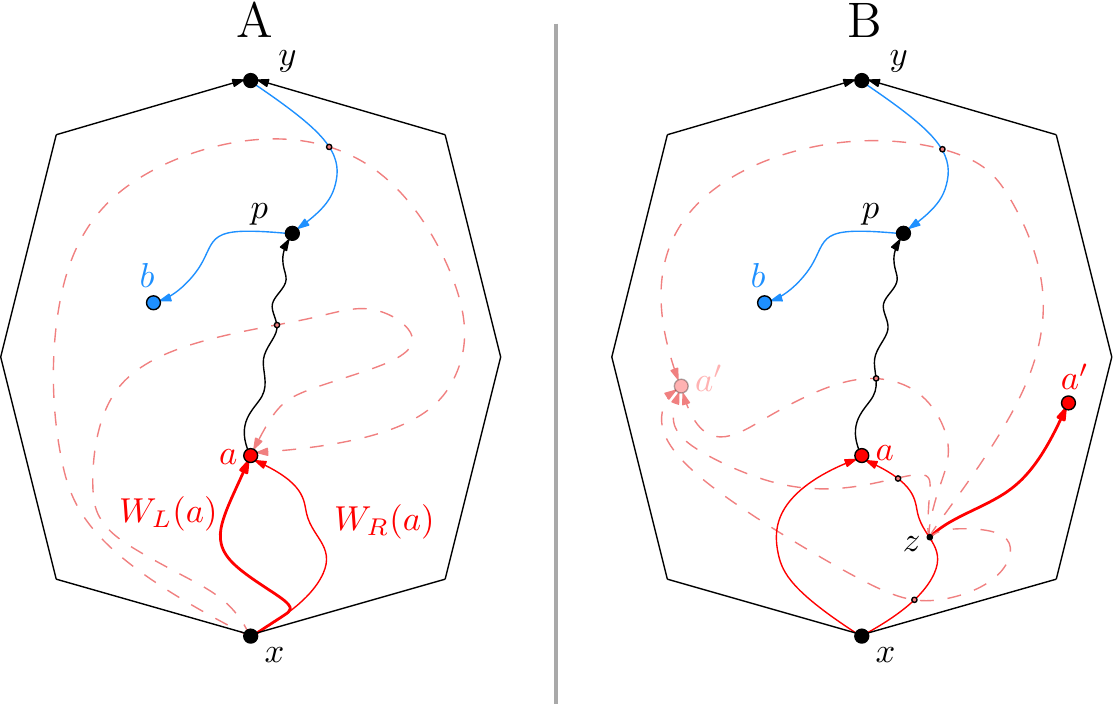} 
  \caption{Dashed lines depict the behavior of considered paths that leads to a contradiction. 
  \newline\textbf{A.} The path $W_L(a)$ has to stay on the left of a separating path.
  \newline\textbf{B.} The element $a'$ has to be on the left of a separating.} 
  \label{fig:sep-paths-obs} 
\end{figure} 
    We only prove the first part -- the second one is symmetric. First, we claim that for every left or right separating path $N$ associated with $a$ and $b$ with the peak $p$ and for each $D \in \{L,R\}$ the only common element of the path $W_D(a)$ and $a[N]y$ is $a$. Indeed, if the paths $W_D(a)$ and $p[N]y$ intersect, then $y \leq_Q a$, contradiction. If $W_D(a)$ and $a[N]p$ intersect, say in element $z$ then $a[N]z$ concatenated with $z[W_D(a)]$ is a directed cycle. See part \textbf{A} of \cref{fig:sep-paths-obs}. This proves the claim.

    Let $M$ be the concatenation of $W_R(a)$ and $a[N]y$. The above claim implies that none of the elements of the path $W_L(a)$ (and obviously $W_R(a)$) is right of $M$. Therefore, if $a'$ is on the right of $M$ then it is on the right of $N$.
    
    The pair $(a,a')$ is a left pair, thus, the first edge $zz'$ of $W_R(a')$ that is not on the path $W_R(a)$ is right of $N$. Suppose that $a'$ is on the left of $M$, then $z' [W_R(a')] a'$ intersects $M$, see part \textbf{B} of \cref{fig:sep-paths-obs}. Let $z''$ be the least common element. If $z''$ belongs to $y [M] p$, then $y \leq_Q a'$, thus, $z''$ lies on $x [M] p$. If $z''$ belongs to the path $x[W_R(a)]z$, then there is a directed cycle. If $z''$ belongs to the path $z[W_R(a)]a - \{z\}$, then a concatenation of $x[W_R(a)]z$, $z[W_R(a')]z''$, and $z''[W_R(a)]a$ is a witnessing path to from $x$ to $a$, which is right of $W_R(a)$, and this contradicts the definition of $W_R(a)$. Finally, if $z''$ lies on the path $a[M]p - \{a\}$, then $a \leq_Q z'' \leq_Q a'$, however, $(a,a')$ was supposed to be an incomparable pair.
\end{proof}

%% file: s.proof.tex
Let $P$ be a cover-planar poset with a unique minimal element $x_0$ such that $\dim(P) \geq 4n+3$ for some $n \geq 2$. 
Fix a planar drawing of the cover graph of $P$ with $x_0$ on the exterior face.
The goal is to find the wheel of order $2n+1$ as a subposet of $P$ and the $n \times n$ grid as a minor of the cover graph of $P$.
\cref{lem:main} applied to $P$ with $k = 2n+1$ gives us a standard example $a_1,\dots,a_k,b_1,\dots,b_k$ within an interval in $P$ satisfying \mbox{items~\ref{lem:main:simple:interval}--\ref{lem:main:see:structure:left:right}}. Let $Q = (Y,\leq_Q)$ be the subposet induced by the interval. 
For each $\alpha\in[2n]$, we choose $M_\alpha := N_R(a_\alpha,b_{\alpha+1})$ with the peak $p_\alpha$, and we choose $N_\alpha := N_L(a_{\alpha+1},b_{\alpha})$ with the peak $q_\alpha$, see part \textbf{A} of \cref{fig:two-chains}. Define
\begin{align*}
    A &:= \{a_\alpha : \alpha \in [2n+1] \}, &&\widehat{A} := \{ z \in Y : z \parallel_Q y \},\\
    B &:= \{b_\alpha : \alpha \in [2n+1] \}, &&\widehat{B} := \{z \in Y : y \leq_Q z\},\\
    C &:= \{p_\alpha : \alpha \in [2n]\}, &&D := \{q_\alpha : \alpha \in [2n]\}.
\end{align*}
We claim that the sets $C$ and $D$ induce chains such that $p_1 <_Q \dots <_Q p_{2n}$ and $q_{2n} <_Q \dots <_Q q_1$, see also part \textbf{B} of \cref{fig:two-chains}.

\begin{figure} 
  \centering 
  \includegraphics{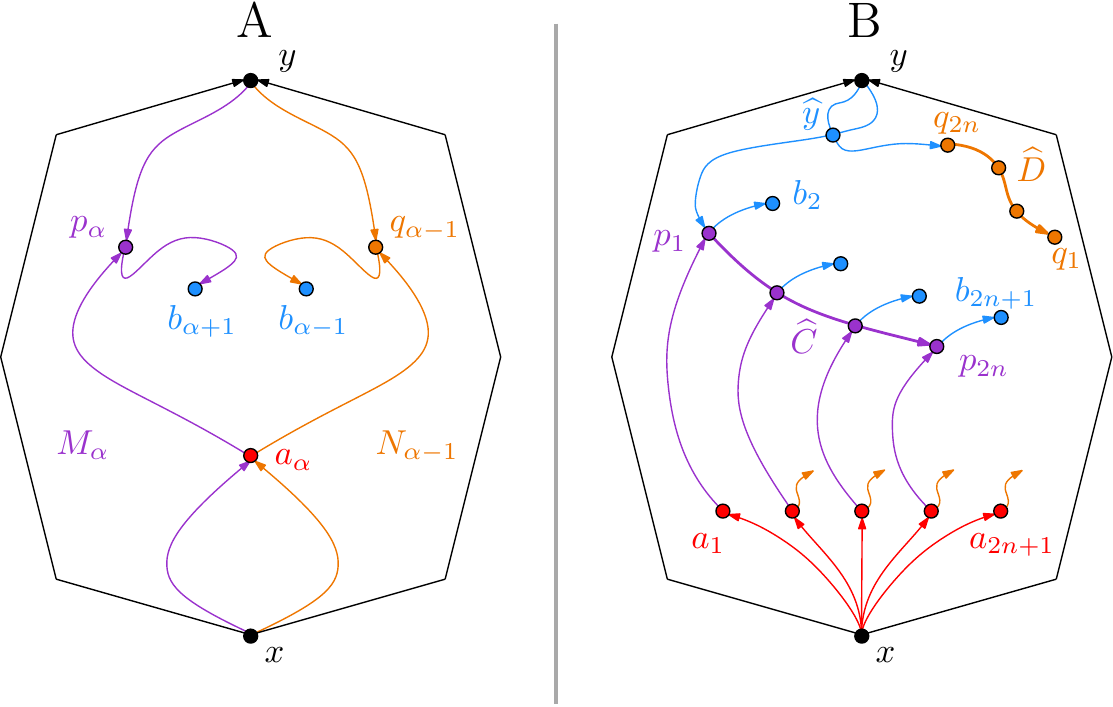} 
  \caption{\newline\textbf{A.} Paths $M_\al$ and $N_\al$. \newline\textbf{B.} A situation asserted in \cref{claim:two-chains} and \cref{claim:C-left-D}.} 
  \label{fig:two-chains} 
\end{figure} 

\begin{claim}\label{claim:two-chains}
    For each $\al \in [2n-1]$ we have $p_\alpha \in W_R(p_{\al+1})$ and $q_{\al+1} \in W_L(q_\al)$. In particular, $p_1,\dots,p_{2n} \in W_R(p_{2n})$ and $q_{2n}, \dots,q_1 \in W_L(q_1)$. Let 
    \begin{align*}
        \widehat{C} := p_1 [W_R(p_{2n})] p_{2n} \ \ \ \text{and} \ \ \ \widehat{D} := q_{2n} [W_L(q_{1})] q_1.
    \end{align*}
     Then, all the pairs in $\widehat{C}\times \widehat{D}$ are incomparable pairs.
\end{claim}
\begin{proof}
    Fix some $\al \in [2n-1]$. First, note that $p_{\al + 1} \not\leq_Q p_\al$, as otherwise $a_{\al+1} \leq_Q p_{\al+1} \leq_Q p_\al \leq_Q b_{\al+1}$. By \cref{obs:separating-paths} applied with the fixed interval, $(a,b) := (a_\al,b_{\al+1})$, $N := M_\al$, and $a' := a_{\al+1}$ (the assumptions of the observation are satisfied by \cref{lem:main}~\ref{lem:main:see:structure}~and~\ref{lem:main:see:structure:left:right}) we obtain that $a_{\al+1}$ is on the right of $M_\al$.

    We claim that $p_{\al+1}$ is right of $M_\al$. Indeed, if $p_{\al+1}$ is not right of $M_\al$, then the paths $a_{\al+1} [N_{\al+1}] p_{\al+1}$ and $M_\al$ intersect. Suppose that they do, and $z$ is a common point. We have $a_{\al+1} \leq_Q z$ and $z \leq_Q p_{\al} \leq_Q b_{\al+1}$, hence $a_{\al+1} \leq_Q b_{\al+1}$, which is false.

    Now we proceed with the proof of $p_\al <_Q p_{\al+1}$. Let $v$ be the greatest common element of $W_R(p_\al)$ and $W_R(p_{\al+1})$. Suppose that $v <_Q p_\al$. The path $W_R(p_\al)$ is a prefix of the path $W_R(b_{\al+1})$, and analogously $W_R(p_{\al+1})$ is a prefix of $W_R(b_{\al+2})$. Therefore, by the fact that the pair $(b_{\al+1},b_{\al+2})$ is a left pair, we obtain that $W_R(p_{\al})$ is left of $W_R(p_{\al+1})$. Let $vu$ be the edge on $W_R(p_{\al+1}$ leaving $v$. It follows that $vu$ is left of $M_\al$. However, we have already proved that $p_{\al+1}$ is right of $M_\al$, therefore, the path $u [W_R(p_{\al+1})] p_{\al+1}$ intersects $M_\al$. Let $z$ be an element in this intersection. If $z$ belongs to $y [M_\al] v$, then $z[M_\al]v$ and $u[W_R(p_{\al+1})]z$ form a directed cycle. Therefore, $z$ lies on $x[M_\al] p_{\al}$ or $v [M_\al] p_{\al}$. The path obtained as the concatenation of the paths $x_0 [W_R(p_{\al+1})] z$ and $z M_\al p_\al$ is a witnessing path from $x$ to $p_\al$, and it is strictly right of $W_R(p_\al)$. The contradiction yields $v = p_\al$, and this gives $p_\alpha \in W_R(p_{\al+1})$. The proof that $q_{\al+1} \in W_L(q_\al)$ is symmetric.
    
    If the set $\widehat{C}\times \widehat{D}$ contains a comparable pair, then either $p_1 \leq_Q q_1$ or $q_{2n} \leq_Q p_{2n}$. Suppose that $p_1 \leq_Q q_1$. The element $q_1$ is the peak of $M_1 = N_L(a_2,b_1)$, thus, $q_1 \leq_Q b_1$. On the other hand $a_1 \leq_Q p_1$. This implies $a_1 \leq_Q b_1$, which is false. Analogously, if $q_{2n} \leq_Q p_{2n}$, then $a_{2n+1} \leq_Q q_{2n} \leq_Q p_{2n} \leq_Q b_{2n+1}$.
\end{proof}

\begin{remark}
    The set $A \cup B \cup C \cup D$ induces the Kelly poset of order $2n+1$ in $P$.
\end{remark}

At this point, we only know that the chains $\widehat{C}$ and $\widehat{D}$ are incomparable. However, it is natural to suspect that $\widehat{C}$ is in some sense left of $\widehat{D}$, see part \textbf{B} of \cref{fig:two-chains}. Let $\widehat{y}$ be the greatest common point of $W_R(p_1)$ and $W_L(q_{2n})$. Note that $y \leq_Q \widehat{y}$. Choose any witnessing path $W$ from $x_0$~to~$\widehat{y}$~in~$P$. We define $W_p$ as the concatenation of $W$ and $\widehat{y}[W_R(p_1)]p_1$, and symmetrically $W_q$ as the concatenation of $W$ and $\widehat{y}[W_L(q_{2n})]q_{2n}$. 

\begin{claim}\label{claim:C-left-D}
    The witnessing path $W_p$ is right of the witnessing path $W_q$. 
\end{claim}
\begin{proof}
    By \cref{claim:two-chains} we have $W_R(p_1) \subset W_R(b_2)$ and $W_L(q_{2n}) \subset W_L(b_2)$. Furthermore, $\widehat{y}$ is the greatest common point of $W_R(p_1)$ and $W_L(q_{2n})$, and belongs to the intersection of $W_R(b_2)$ and $W_L(b_2)$. This yields the claim by the definition of the leftmost and rightmost witnessing paths.
\end{proof}

For each $\alpha\in[2n]$, let $U_\al$ be the witnessing path from $a_\al$ to $p_\al$ that was chosen to be a part of $M_\alpha = N_R(a_\al,b_{\al+1})$. Analogously, let $V_\al$ be the witnessing path from $a_{\al+1}$ to $q_\al$ that was chosen to be a part of $N_\alpha = N_L(a_{\al+1},b_{\al})$.

We say that a witnessing path $W$ \emph{crosses in order} a tuple of witnessing paths $(W_1,\dots,W_s)$ if $W$ intersects each path in the tuple and for every $i \in [s-1]$ all common points of $W$ with $W_i$ are less in $Q$ than all common points of $W$ with $W_{i+1}$. For two witnessing paths $W,W'$ we say that $W$ is \emph{over} $W'$ if there does not exist $w \in W$ and $w' \in W'$ such that $w \leq_Q w'$. In particular, if one path is over another one, then they are disjoint. We say that $W,W'$ are \emph{incomparable} if $W$ is over $W'$ and $W'$ is over $W$.

\begin{figure} 
  \centering 
  \includegraphics{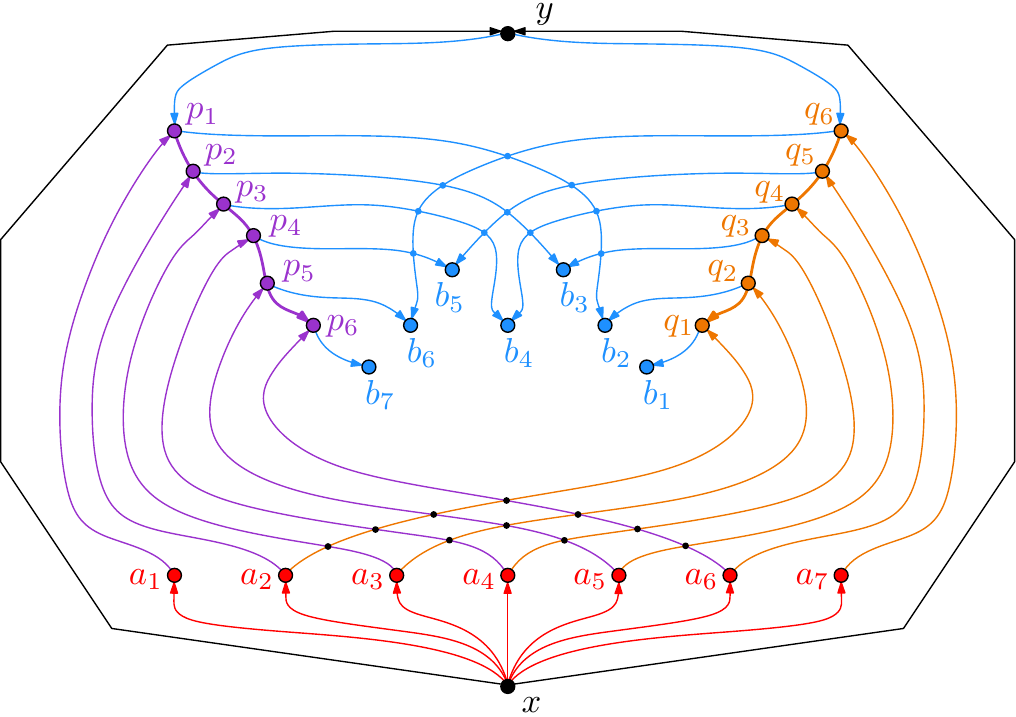} 
  \caption{ 
  The general idea of the proof is to show that the standard example produced by \cref{lem:main} admits a similar structure as on the figure.
  } 
  \label{fig:final} 
\end{figure} 
\vbox{
\begin{claim}\label{claim:UV}
    Let $\al \in [2n]$. The witnessing path $U_\al$ has the following properties (see \cref{fig:final}).
    \begin{enumerateNumU}
        \item $U_\al$ is over the paths $U_1,U_2,\dots,U_{\al-1}$.\label{u1}
        \item $U_\al$ crosses in order $(V_{\al-1},V_{\al-2},\dots,V_1)$.\label{u2}
        \item $U_\al$ is incomparable with each of the paths $V_\al,V_{\al+1},\dots,V_{2n}$.\label{u3}
    \end{enumerateNumU}
     Analogously, the witnessing path $V_\al$ has the following properties.
    \begin{enumerateNumV}
        \item $V_\al$ is over the paths $V_{\al+1},V_{\al+2},\dots,V_{2n}$.\label{v1}
        \item $V_\al$ crosses in order $(U_{\al+1},U_{\al+2},\dots,U_{2n})$.\label{v2}
        \item $V_\al$ is incomparable with each of the paths $U_1,U_2,\dots,U_{\al}$.\label{v3}
    \end{enumerateNumV}
\end{claim}
}
\begin{proof}
    To prove item~\ref{u1} suppose for a contradiction that there exist $\beta \in [\al-1]$ and $u \in U_\al, u' \in U_\beta$ such that $u \leq_Q u'$. We have
        \[ a_\al \leq_Q u \leq_Q u' \leq_Q p_\beta \leq_Q p_{\al-1} \leq_Q b_{\al}. \]
    The fourth comparability follows from \cref{claim:two-chains}. This is a contradiction. Analogously, we obtain item~\ref{v1}.

    Note that if the path $U_\al$ crosses each of the paths $V_{\al-1},V_{\al-2},\dots,V_1$, then we almost immediately obtain that it crosses them in order. Indeed, suppose that there exist $1 \leq \beta < \gamma \leq \al-1$ and $v \in V_{\gamma},v' \in V_{\beta}$ such that $v,v' \in U_\al$ and $v' \leq_Q v$. However, by item~\ref{v1} the path $V_{\beta}$ is over the path $V_{\gamma}$, which is a contradiction.
    
    Now, we will prove that the path $U_\al$ truly crosses each of the paths $V_{\al-1},V_{\al-2},\dots,V_1$. The element $a_\al$ is a common point of $U_\al$ and $V_{\al-1}$. Fix $\beta \in [\al-2]$. First, we argue that $q_\beta$ is on the right of $M_\al$. If $q_\beta$ lies on $M_\al$, then we have $q_\beta \leq_Q p_\al$, which is false due to \cref{claim:two-chains}. Suppose that $q_\beta$ is on the left of $M_\al$. Let $\widehat{y}u$ be the edge leaving $\widehat{y}$ in $W_L(q_\beta)$. By \cref{claim:C-left-D}, $\widehat{y}u$ is on the right of $M_\al$. It follows that the path $\widehat{y} [W_L(q_\beta)] q_\beta$ intersects the path $M_\al$. Let $z$ be an element in this intersection. We consider all the cases of the position of the element $z$ on the path $M_\al$, see part \textbf{A} of \cref{fig:claim-8}.

    If $z \in y [M_\al] \widehat{y}$, then there is a directed cycle in $P$.
    If $z \in \widehat{y}[M_\al] p_1 \backslash \{\widehat{y},p_1\}$, then we obtain a contradiction with the definition of $\widehat{y}$.
    If $z \in p_1 [M_\al] p_\al$, then we have $p_1 \leq_Q z \leq_Q q_\beta$, which is false because $\widehat{C}$ and $\widehat{D}$ are incomparable.
    Suppose that $z \in a_\al [M_\al] p_\al$. By \cref{claim:two-chains} we have $q_{2n} \in W_L(q_\beta)$. If $q_{2n} \leq_Q z$, then $q_{2n} \leq_Q p_\al$, thus, $z <_Q q_{2n}$. Therefore, $a_\al \leq_Q z <_Q q_{2n} <_Q q_{\al} \leq_Q b_\al$, which is false.
    If $z \in x [M_\al] a_\al$, then $y <_Q z \leq_Q a_\al$, which is false by \cref{lem:main}~\ref{lem:main:see:structure}.   

    We proved that $q_\beta$ is on the right of $M_\al$. By \cref{obs:separating-paths} applied to the fixed interval, $(a,b) = (a_\al,b_{\al+1})$, $N = M_\al$, and $a' = a_{\beta+1}$ (the pair $(a_\al,a_{\beta+1})$ is an incomparable right pair due to \cref{lem:main}~\ref{lem:main:see:structure:left:right}) we obtain that $a_{\beta+1}$ is on the left of $M_\al$. It follows that the path $V_\beta$ (connecting $a_{\beta+1}$ with $q_\beta$) intersects the path $M_\al$. Let $z$ be an element in this intersection. Again we consider all the cases of the position of the element $z$ on the path $M_\al$, see part \textbf{B} of \cref{fig:claim-8}.
    
    If $z\in x [M_\al] a_\al$, then $a_{\beta+1} \leq_Q a_\al$, which is false.
    If $z\in a_\al [M_\al] p_\al = U_\al$, then the claim is proved.
    If $z \in p_1 [M_\al] p_\al$, then $p_1 \leq_Q q_\beta$, which is false.
    If $z \in y [M_\al] p_1$, then $a_{\beta+1} \leq_Q z \leq_Q p_1 \leq_Q p_{\beta} \leq_Q b_{\beta+1}$, which is false.
    
    This concludes the proof that $U_\al$ crosses $V_\beta$, and the whole item~\ref{u2}. The proof of item~\ref{v2} is symmetric.

\begin{figure} 
  \centering 
  \includegraphics[scale=1]{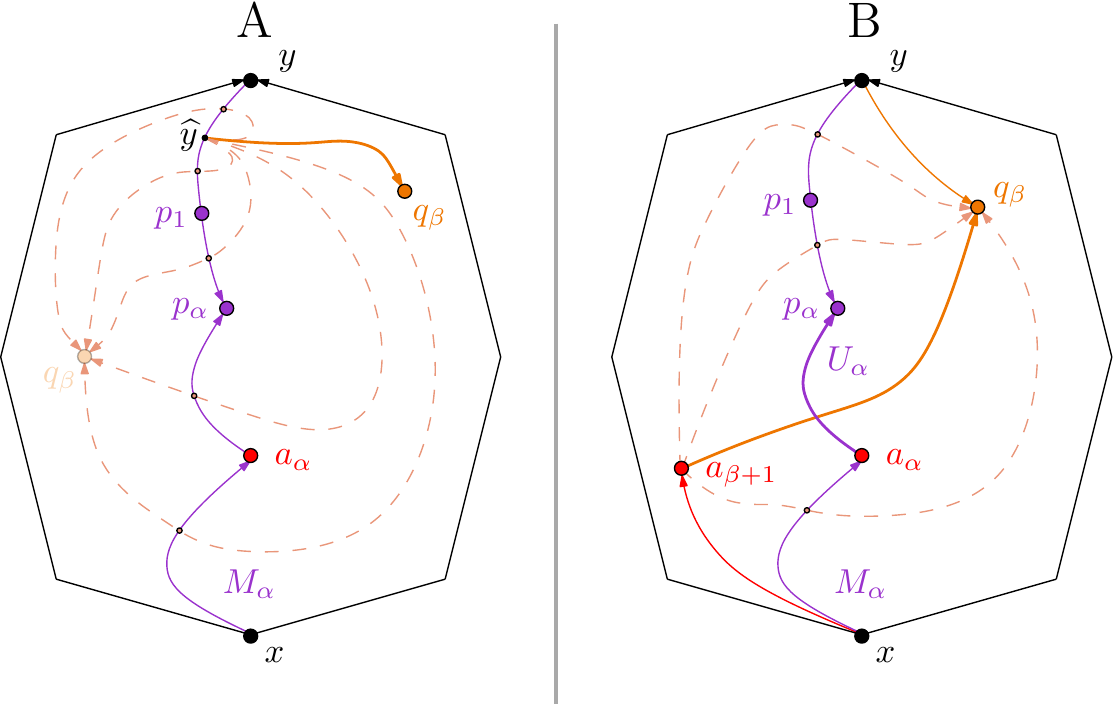} 
  \caption{Dashed lines depict the behavior of considered paths that leads to a contradiction. 
  \newline\textbf{A.} The element $q_\beta$ has to be right of $M_\al$. \newline\textbf{B.} The path $V_\beta$ has to intersect $U_\al$.} 
  \label{fig:claim-8} 
\end{figure} 

    Finally, we will prove \ref{u3}. Suppose that $U_\al$ is not incomparable with $V_\beta$ for some $\beta \in [\alpha,2n]$. First, suppose that there exist $u \in U_\al$ and $v \in V_\beta$ such that $u \leq_Q v$. Then by \cref{claim:two-chains} we have
        \[a_\al \leq_Q u \leq_Q v \leq_Q q_\beta \leq_Q q_\al \leq_Q b_\al.\]
    This is false, therefore, there exists $u \in U_\al$ and $v \in V_\beta$ such that $v \leq_Q u$. However, then
        \[a_{\beta+1} \leq_Q v \leq_Q u \leq_Q p_\al \leq_Q p_{\beta} \leq_Q b_{\beta+1}.\]
    The contradiction gives us \ref{u3} (and \ref{v3} by symmetry).
\end{proof}

    We are ready to prove \cref{thm:nxn-grid}, namely, that $G$ contains $n \times n$ grid as a minor.

    Let $\calP := \{V_\al : \al \in [2n-1]\} \cup \{U_\beta : \beta \in [2,2n]\}$. We start by removing all edges and vertices that do not belong to any path in $\calP$ obtaining a new graph $H$, which is a minor of $G$. Fix some $\al$ and $\beta$ such that $1 \leq \al < \beta \leq 2n$. By \cref{claim:UV}~\ref{u2}~and~\ref{v2} the paths $V_\al$ and $U_\beta$ intersect. Both paths induce chains in the poset, thus, the intersection is also a chain. Let $s_{\al,\beta}$ be the least element in the intersection and $t_{\al,\beta}$ be the greatest.
    
\begin{proof}[Proof of \cref{thm:nxn-grid}]
    Define 
        \[S_{\al,\beta} := \{v \in V(H) : v \in s_{\al,\beta} [V_\al] t_{\al,\beta} \cup s_{\al,\beta} [U_\beta] t_{\al,\beta}\}.\]
    Clearly, the graph $H[S_{\al,\beta}]$ is connected. The structure of the crossings of paths in $\calP$ obtained in \cref{claim:UV} yields that for every $1 \leq \al' < \beta' \leq 2n$ such that $(\al,\beta) \neq (\al',\beta')$ we have $S_{\al,\beta} \cap S_{\al',\beta'} = \emptyset$. Also note that each vertex of $H$ that is not a member of any $S_{\al,\beta}$ belongs to exactly one path in $\calP$, and thus is of degree at most $2$. For each $1 \leq \al < \beta \leq 2n$ we can contract the set $S_{\al,\beta}$ obtaining a new graph $H'$, which is still a minor of $G$. Clearly this graph is a subdivision of the $n\times n$ grid, which ends the proof.
\end{proof}

To conclude the proof of \cref{thm:wheel}, we need to force the grid-like structure also in the upper part of the drawing, see \cref{fig:final}. To this end, we prove the following.
\begin{claim}\label{claim:z-intersection}
  If $1<\al<\beta<2n+1$, then the paths $p_{\al-1}[W_R(b_\al)]b_\al$ and $q_{\beta}[W_L(b_\beta)]b_\beta$ intersect. 
\end{claim}
\begin{proof}
  Consider the region $\calR$ enclosed by $\widehat{y}[W_R(b_\al)]b_\al$ and $\widehat{y}W_L(b_\al)b_\al$. Note that for each $z$ on the boundary of $\calR$ we have $y \leq_Q z \leq_Q b_\al$.
  Moreover, for each $z$ lying in $\calR$ we have $y \leq_Q z$. Indeed, $x$ is not in $\calR$ and $x <_Q z$, it follows that a witnessing path from $x$ to $z$ crosses the boundary of $\calR$, and hence $y \leq_Q z$. The element $a_\al$ is outside of $\calR$ because $a_\al \parallel_Q y$. If $p_{\beta-1}$ was in $\calR$, then the path $a_\al [M_\al] p_\al$ concatenated with the path $p_\al [\widehat{C}] p_{\beta-1}$ would cross the boundary of $\calR$ resulting in $a_\al <_Q b_\al$, which is a contradiction. Therefore, $p_{\beta-1}$ is outside of $\calR$. Moreover, by the fact that $a_\al <_Q p_{\beta-1}$, any witnessing path from $p_{\beta-1}$ to $b_\beta$ do not intersect the boundary of $\calR$, which implies that $b_\beta$ is outside of $\calR$. 
  
  Consider the path $W := q_\beta[W_L(b_\beta)]b_\beta$. By \cref{claim:two-chains} the element $q_\beta$ lies on the boundary of $\calR$. Recall that $(b_\al,b_\beta)$ is a left pair (by \cref{lem:main}~\ref{lem:main:see:structure:left:right}). Therefore, the first edge of $W$ that is not on $W_L(b_\al)$ is right of $W_L(b_\al)$. Note that this edge is inside $\calR$. However, $b_\beta$ is outside $\calR$, thus $W$ intersects the boundary of $\calR$. Let $z$ be the least element in the intersection of $W$ and the boundary of $\calR$. If $z \in \widehat{y}[W_L(b_\al)]q_\beta$, then we have a directed cycle. If $z \in q_\beta[W_L(b_\al)]b_\al\backslash \{q_\beta\}$, then the concatenation of
    \[x_0[W_L(b_\al)] q_\beta, q_\beta [W] z, z[W_L(b_\al)]\]
  is a witnessing path from $x_0$ to $b_\al$, which is strictly left of $W_L(b_\al)$, and contradicts the definition of $W_L(b_\al)$. If $z \in \widehat{y}[W_R(b_\al)]p_{\al-1}$, and $z$ is an element in the intersection, then $q_\beta <_Q z <_Q p_{\al-1}$, which contradicts \cref{claim:two-chains}. Therefore, we obtain that $W = q_\beta[W_L(b_\beta)]b_\beta$ intersects $p_{\al-1}[W_R(b_\al)]b_\al$.
\end{proof}

For $1<\al<\beta<2n+1$, let $z_{\al,\beta}$ be the least point in the intersection of the paths $p_{\al-1}[W_R(b_\al)]b_\al$ and $q_{\beta}[W_L(b_\beta)]b_\beta$ (it is well-defined by \cref{claim:z-intersection}).

\begin{claim}\label{claim:z-ordering}
    Let $\al \in [2,2n-1]$ and $\beta \in [3,2n]$.
    \begin{enumerateNumZ}
        \item For every $\gamma \in [\al+1,2n]$ we have $z_{\al,\gamma+1} <_Q z_{\al,\gamma}$.\label{z1}
        \item For every $\gamma \in [3,\beta-1]$ we have $z_{\gamma-1,\beta} <_Q z_{\gamma,\beta}$.\label{z2}
    \end{enumerateNumZ}
\end{claim}
\begin{proof}
    Let $\gamma \in [\al+1,2n]$. Both elements $z_{\al,\gamma+1}$ and $z_{\al,\gamma}$ lie on the path $p_{\al-1}[W_R(b_\al)]b_\al$, thus they are comparable. Suppose that $z_{\al,\gamma} \leq_Q z_{\al,\gamma+1}$. Then,
        \[ a_{\gamma+1} \leq_Q q_{\gamma} \leq_Q z_{\al,\gamma} \leq_Q z_{\al,\gamma+1} \leq_Q b_{\gamma+1}.\]
    This is a contradiction, thus $z_{\al,\gamma+1} <_Q z_{\al,\gamma}$, and item~\ref{z1} is proved. The proof of item~\ref{z2} is symmetric.
\end{proof}

Now, we are ready to prove the main result of this paper, namely, that $Q$ contains the wheel of order $2n+1$ as a subposet. To this end recall that the wheel of order $2n+1$ is a poset on the ground set  $\{r_{i,j} : i,j \in [2n+1], j+1 \not\equiv i \mod 2n+1\}$ and $r_{i,j} \leq r_{i',j'}$ in the wheel if and only if $\cycl{i',j'} \subset \cycl{i,j}$. For simplicity, we will write that $i,j$ are \emph{admissible} if $i,j \in [2n+1]$ and $j+1 \not\equiv i \mod 2n+1$.

\begin{proof}[Proof of \cref{thm:wheel}]
    We define
    \begin{itemize}
        \item $r_{\beta,\al} := s_{\al,\beta-1}$ for each $\al,\beta \in [2n+1]$ with $\al+1 < \beta$,
        \item $r_{\al,\al} := b_\al$ for each $\al \in [2n+1]$,
        \item $r_{\al,\beta} := z_{\al,\beta}$ for each $1<\al<\beta<2n+1$,
        \item $r_{1,\beta} := q_{\beta}$ for each $1<\beta<2n+1$,
        \item $r_{\al,2n+1} := p_{\al-1}$ for each $1<\al<2n+1$.
    \end{itemize}
    By the combination of all already established claims for every admissible $i,j$ we have $r_{i,j} \leq_Q b_\gamma$ if and only if $\gamma \in \cycl{i,j}$. Therefore, if $\cycl{i',j'} \not\subset \cycl{i,j}$, then $r_{i,j} \parallel_Q r_{i',j'}$. It suffices to show that if $\cycl{i',j'} \subsetneq \cycl{i,j}$, then $r_{i,j} <_Q r_{i',j'}$. By transitivity, it suffices to show the above only in the case where $|\cycl{i,j}\backslash \cycl{i',j'}| = 1$. We will study this case by case.

    Let $1 < \al \leq \beta < 2n$. By \cref{claim:z-ordering}~\ref{z1}, we have $r_{\al,\beta+1} <_Q r_{\al,\beta}$.

    Let $2 < \al \leq \beta < 2n + 1$. By \cref{claim:z-ordering}~\ref{z2} we have $r_{\al-1,\beta} <_Q r_{\al,\beta}$.

    Let $\al = 2$ and $2 < \beta < 2n+1$. Then, $r_{1,\beta} \leq_Q r_{2,\beta}$ by definition. If $q_\beta = r_{1,\beta} = r_{2,\beta} = z_{2,\beta}$, then $p_1 \leq_Q q_\beta$, which contradicts \cref{claim:two-chains}.

    Let $1<\al<2n+1$ and $\beta = 2n$. Then, $r_{\al,2n+1} \leq_Q r_{\al,2n}$ by definition. If $p_{\al-1} = r_{\al,2n+1} \leq_Q r_{\al,2n} = z_{\al,2n}$, then $q_{2n} \leq_Q p_{\al-1}$, which contradicts \cref{claim:two-chains}.

    Let $\al = 1$ and $2 < \beta < 2n+1$. We want to prove that $r_{2n+1,\beta} <_Q r_{1,\beta}$. In other words $s_{\beta,2n} <_Q q_{\beta}$. It is clear by definition that $s_{\beta,2n} \leq_Q q_{\beta}$. If the comparability is not strict, then we have $q_\beta = s_{\beta,2n} \leq_Q p_{2n}$, which contradicts \cref{claim:two-chains}.

    Let $1 < \al < 2n+1$ and $\beta = 2n+1$. We want to prove that $r_{\al,1} <_Q r_{\al,2n+1}$. In other words $s_{1,\al-1} <_Q p_{\al-1}$. It is clear by definition that $s_{1,\al-1} <_Q p_{\al-1}$. If the comparability is not strict, then we have $p_{\al-1} = s_{1,\al-1} \leq_Q q_{1}$, which contradicts \cref{claim:two-chains}.

    Let $\al,\beta \in [2n+1]$ with $\al+2 < \beta$. By \cref{claim:UV}~\ref{v2} we have $r_{\beta+1,\al} = s_{\al,\beta} <_Q s_{\al,\beta-1} = r_{\beta,\al}$. And, by \cref{claim:UV}~\ref{u2} we have $r_{\beta,\al-1} = s_{\al-1,\beta-1} <_Q s_{\al,\beta-1} = r_{\beta,\al}$.
\end{proof}

%% file: s.open-problems.tex
Let $m(P)$ be the number of minimal elements of $P$ and $t(P)$ be the treewidth of the cover graph of $P$. We proved that for a cover-planar poset $P$ we have $\dim(P) \leq m(P)(4t(P)+6)$. It is natural to ask if this inequality is asymptotically tight. For the wheels we have $\dim(H_d)$ being $\Omega(t(H_d))$. On the other hand, for the Kelly posets we have $\dim(K_d)$ being $\Omega(m(K_d))$. Therefore, in general, $\dim(P)$ is $\Omega(t(P) + m(P))$ among cover-planar posets. We conjecture that this is not a matching lower bound.
\begin{conjecture}
    Among posets cover-planar posets, we have $\dim(P)$ is $\Omega(m(P)t(P))$.
\end{conjecture}

We also proved that among cover-planar posets with unique minimal elements large dimension forces a large wheel number. Interestingly, we do not know any essentially different constructions of cover-planar posets with large standard example number and a unique minimal element. Therefore, we conjecture the following.

\begin{conjecture}
    For every poset $P$ with a unique minimal element and a planar cover graph, we have
        \[\se(P) = \wheel(P).\]
\end{conjecture}

Note that the above is not true (even with an additive factor) if we drop the assumption on a unique minimal element. In \cite[Theorem~3]{JMW17} authors constructed a cover-planar poset $P_h$ for each positive $h$ such that the height of $P_h$ is $h$ and $\se(P_h) \geq 2h-2$. Recall that the height of the wheel $H_d$ is equal to $d$. Therefore, $\wheel(P_h) \leq h$. Finally, one can attach two elements to $P_h$ in such a way that $P_h$ has exactly two minimal elements, and the height increases by at most one.